\documentclass[12pt]{amsart}
\usepackage{amsmath,amssymb,latexsym,tikz}
\usepackage{mathrsfs,amsfonts}
\usepackage[colorlinks=true,urlcolor=blue,
citecolor=red,linkcolor=blue,linktocpage,pdfpagelabels,
bookmarksnumbered,bookmarksopen]{hyperref}
\usepackage[english]{babel}

\usepackage[left=3cm,right=3cm,top=3cm,bottom=3cm]{geometry}

\usepackage[hyperpageref]{backref}
\usepackage{mathrsfs}
\usepackage{lmodern}

\numberwithin{equation}{section}

\newtheorem{theorem}{Theorem}[section]
\newtheorem{lemma}[theorem]{Lemma}

\newtheorem{proposition}[theorem]{Proposition}

\theoremstyle{definition}

\newcommand{\R}{{\mathbb R}}

\newcommand{\eps}{\varepsilon}

\newcommand{\weakstarto}{\stackrel{*}{\rightharpoonup}}

\newcommand{\pnorm}[2][]{\if #1'' \left|#2\right|_p \else \left|#2\right|_{#1} \fi}

\newcommand{\cal}{\mathcal}

\def\XXint#1#2#3{{\setbox0=\hbox{$#1{#2#3}{\int}$ }
\vcenter{\hbox{$#2#3$ }}\kern-.6\wd0}}

\title[Quantitative truncation estimates for Hardy-Sobolev optimizers]{Quantitative truncation estimates for fractional Hardy-Sobolev optimizers}

\author[S.A. Marano]{S.~A.\,  Marano}
\address[S.A. Marano]{Dipartimento di Matematica e Informatica
\newline\indent
Universit\`a degli Studi di Catania
\newline\indent
Viale A. Doria 6 I-95125 Catania, Italy}
\email{marano@dmi.unict.it}

\author[S. Mosconi]{S.\,  Mosconi}
\address[S. Mosconi]{Dipartimento di Matematica e Informatica
\newline\indent
Universit\`a degli Studi di Catania
\newline\indent
Viale A. Doria 6 I-95125 Catania, Italy}
\email{mosconi@dmi.unict.it}



\thanks{Work performed within PTR 2018-2020 - linea di intervento 2: `Metodi Variazionali ed Equazioni Differenziali' of the University of Catania and partly funded by Research project of MIUR (Italian Ministry of Education, University and Research) Prin 2017 `Nonlinear Differential Problems via Variational, Topological and Set-valued Methods' (Grant Number 2017AYM8XW)}

\begin{document}
\begin{abstract}
The general stability problem of truncations for a family of functions concentrating mass at the origin is described and a concrete example in the framework of entire optimizers for the fractional Hardy-Sobolev inequality is given. In this short note we point out some quantitative stability estimates, useful in dealing with critical $p-q$ fractional equations. 
\end{abstract}
\maketitle
\section{Introduction and main results}

In the last years, a great deal of research has grown around multi-dimensional fractional differential problems of the form
\begin{equation}\label{model}
{\cal K} u=f(x,u),
\end{equation}
where ${\cal K}$ denotes a suitably defined elliptic fractional non-local operator. A general model for {\em linear} ${\cal K}$ is 
\[
{\cal K} u (x)={\sf p.\,v.}\int_{\R^{N}}K(x, y) (u(x)-u(y))\, dy, \qquad K(x, y)\simeq |x-y|^{N+\sigma},
\]
while the main example in the nonlinear setting reads as
\begin{equation}
\label{def}
(-\Delta_{p})^{s}u:=\frac{1}{p}{\sf d}[u]_{s,p}^{p},\qquad [u]_{s, p}^{p}=\iint_{\R^{N}\times\R^{N}}\frac{|u(x)-u(y)|^{p}}{|x-y|^{N+p\, s}}\, dx\, dy\, ,
\end{equation}
with ${\sf d}$ being the Fr\'echet differential. Both families encompass the celebrated fractional Laplacian as a special case. 

A large part of this research concerns existence and multiplicity of solutions to \eqref{model}. Indeed, modern non-linear analysis provides a lot of relatively abstract machinery to get such kind of results, and the general schemes of proof usually work once two sets of conditions are met. The first can be called {\em lower order} set of assumptions, as it mainly relates to the right-hand side of \eqref{model}, having little to do with the nature of the driving operator, except for the parameters that define it. Some examples are sub-criticality, sub/super-linearity, or Ambrosetti-Rabinowitz conditions, which are often explicitly imposed in the literature. The second set of conditions is the {\em leading order} one, and it pertains (sometimes subtle) properties of ${\cal K}$ alone, such as the corresponding regularity theory or relevant functional analytic embeddings. Needless to say, the more interesting applications of non-linear analysis in the fractional framework are those where some leading order assumption fails. Indeed, the true nature of
${\cal K}$ lies in what distinguishes it from the usual elliptic differential operators, and an extended discussion of such differences as well as related literature can be found in \cite{MSP}.

Let us now describe a meaningful feature of problems such as \eqref{model} from the functional analytic point of view. If $\Omega$ is a smooth subset of $\R^{N}$ and $0<s<1<p<N/s$ then $(-\Delta_{p})^{s}$ naturally acts on the fractional Sobolev space $W^{s, p}_{0}(\Omega)$, namely the space of all measurable $u:\R^{N}\to \R$ supported in $\overline\Omega$, vanishing at infinity
\footnote{which means $|\{x\in\Omega:|u(x)|>\eps\}|<+\infty$ for every $\eps>0$. This additional condition caters technical issues when $\Omega$ is unbounded.},
and such that the norm $[u]_{s,p}$ in \eqref{def} is finite. In many aspects, the parameter $s$ plays the r\^ole of a differentiability scale, while $p$ prescribes the summability of the $s$-fractional derivative. A suggestive notation giving meaning to this statement consists in defining the $s$-fractional incremental ratio of $u$ and the singular measure $\mu$ on $\R^{2N}$ as
\[
|D^{s}u|(x, y)=\frac{|u(x)-u(y)|}{|x-y|^{s}},\qquad d\mu=|x-y|^{-N}\, dx\, dy,
\]
respectively, so that
\[
[u]_{s, p}=\|D^{s}u\|_{L^{p}(\R^{2N}, d\mu)}.
\]
Notice that since $\mu(\R^{2N})=+\infty$, $L^{p}(\R^{2N}, d\mu)$ does not embed into $L^{q}(\R^{2N}, d\mu)$ for $q<p$ and, accordingly, $W^{s, p}_{0}(\R^{N})$ does not embed into $W^{s,q}_{0}(\R^{N})$. So far so good, as the latter embedding also fails for classical (non-fractional) Sobolev spaces. However, when $\Omega$ is bounded, H\"older's inequality entails
\begin{equation}\label{emb}
W^{1,p}_{0}(\Omega)\hookrightarrow W^{1,q}_{0}(\Omega)\qquad \text{for any $p> q$},
\end{equation}
which led H. Br\'ezis to ask whether such an embedding holds true also at the fractional level for {\em bounded} smooth domains. Surprisingly enough, Mironescu and Sickel \cite{MS} proved that the fractional Sobolev version of \eqref{emb} actually fails even in a set theoretic sense. Notice that $H^{s, p}_{0}(\Omega)$ actually embeds into $H^{s, q}_{0}(\Omega)$ if $H^{s, p}$ denotes the fractional Bessel potential space. So, when $p> q$, the mixed energy functional
\begin{equation}\label{spq}
J(u)=\| D^{s}u\|_{L^{p}(\R^{2N}, d\mu)}^{p}+\|D^{s}u\|_{L^{q}(\R^{2N}, d\mu)}^{q}
\end{equation}
is well defined and smooth in a space which is smaller than $W^{s, p}_{0}(\Omega)$ and is therefore more delicate to treat with respect to the classical one
\[
J(u)=\|D u\|_{L^{p}(\R^{N})}^{p}+\|D u\|_{L^{q}(\R^{N})}^{q}.
\]
In the non-fractional case $J$ gives rise to the so-called {\em $p$-$q$ Laplacian}, which serves as a model for many applications; see the survey \cite{MM1} and the references therein. The previous discussion highlights that studying its fractional counterpart (given by the differential of \eqref{spq}) requires more care and, sometimes, completely different techniques.

A meaningful item is the problem of {\em quantitative truncation estimates}. Let us describe it in broad (and somehow vague) terms.  Given a function space $X\subseteq L^{1}_{{\rm loc}}(\R^{N})$, (that is $W^{s,p}_{0}(\R^{N})\cap W^{s,q}_{0}(\R^{N})$ in our example), consider a family of non-negative functions $\{U_{\eps}\}_{\eps}\subseteq X$ concentrating at the origin, i.e., $U_{\eps}\, dx\weakstarto\mu$ as $\eps\to 0^+$, with $\mu\preccurlyeq \delta_{0}$. A {\em truncation} of $\{U_{\eps}\}_{\eps}$ in  $B_{\delta}$ is a family $\{U_{\eps, \delta}\}_{\eps}\subseteq X$ fulfilling
\begin{equation}
\label{ctrunc}
{\rm supp}(U_{\eps, \delta})\subseteq B_{2\delta},\qquad {\rm supp} (U_{\eps, \delta}-U_{\eps})\subseteq \R^{N}\setminus B_{\delta}.
\end{equation}
A {\em quantitative truncation estimate} for a functional $I:X\to\R$ is an explicit first-order asymptotic analysis, as $\eps,\delta\to 0^+$ of $I(U_{\eps, \delta})$: one usually defines $I_{0}$ taking appropriate limits of $I(U_{\eps, \delta})$ and aims at finding explicit bounds (from below, above, or both) for $I(U_{\eps, \delta})- I_{0}$.  Clearly, there are many ways to truncate a family of concentrating functions, and each one produces, in principle, different truncation estimates. When $X$ is a $C^{\infty}_{c}(\R^{N})$-modulus, the {\em truncation by multiplication} looks the most natural: pick any $\varphi\in C^{\infty}_{c}(B_{2})$ such that $\varphi\equiv 1$ on $B_{1}$ and put
\[
U_{\eps,\delta}(x)=\varphi\left(\frac{x}{\delta}\right)\, U_{\eps}(x).
\]
Sometimes the second condition in \eqref{ctrunc} can be weakened or even completely dropped, and general projection operators $\pi_{\delta}: X\to X_{\delta}$, where $X_\delta=\{u\in X:{\rm supp}(u)\subseteq B_{2\delta}\}$, considered. We will not dwell on details of other methods, but rather focus on the particular concrete setting we are going to investigate. 

The fractional Hardy-Sobolev inequality reads as
\begin{equation}\label{HS}
\left(\int_{\R^N} \frac{|u|^r}{|x|^\alpha}\, dx\right)^{\frac{1}{r}}\leq C\left(\int_{\R^{2N}}|D^{s}u|^{p}\, d\mu\right)^{\frac{1}{p}}.
\end{equation}
Here, $p>1$, $s\in \, ]0, 1[$, $0\le\alpha\le p\, s<N$, and $r$ is dictated by scaling through
\begin{equation}\label{scal}
\frac{N-\alpha}{r}=\frac{N-p\, s}{p}.
\end{equation}
Every function that realizes the optimal constant in \eqref{HS} is called an Aubin-Talenti function. By analogy with the local case, which formally corresponds to $s=1$, it is conjectured that the Aubin-Talenti functions, up to constant multiples, rescaling and possible (in the case $\alpha=0$) translations, are
\begin{equation}\label{AT}
 U(x)=(1+|x|^{\frac{p-\alpha/s}{p-1}})^{\frac{p\, s-N}{p-\alpha/s}}.
 \end{equation}
If $\alpha<ps$ then they can be obtained by solving the minimization problem
\begin{equation}\label{S}
0<{\cal S}=\inf\left\{\frac{[u]_{s,p}^p}{\|u\|_{r,\alpha}^{p}}:0<\|u\|_{r,\alpha}<+\infty\right\},\quad\text{where}\; \|u\|_{r, \alpha}=\left(\int_{\R^N} \frac{|u|^r}{|x|^\alpha}\, dx\right)^{\frac{1}{r}}
\end{equation}
via concentration-compactness. Some basic properties of the minimizers are described below.
\begin{proposition}[\cite{MM}, Theorem 1.1]
\label{propmin}
Let $p>1$, $s\in \ ]0, 1[$, $0\le \alpha<p\, s<N$, and $r$ satisfy \eqref{scal}. Then \eqref{S}  is solvable and its minimizers $U$ are bounded continuous functions of strict constant sign. The positive ones turn out (eventually after translation in the case $\alpha=0$) radial and radially non-increasing. They obey the decay estimate
\begin{equation}\label{decay}
U(\rho)\simeq \rho^{-\frac{N-p\, s}{p-1}} \qquad \text{as $\rho\to +\infty$}
\end{equation}
and, moreover,
\begin{equation}\label{decD}
[U]_{s,q}=\|D^{s}U\|_{L^{q}(\R^{2N}, d\mu)}<+\infty\quad\forall\, q\in \left]\frac{N(p-1)}{N-s}, p\right].
\end{equation}
\end{proposition}
Since problem \eqref{S} is homogeneous in $u$, the set of its positive minimizers turns out to be a cone. The associated Euler-Lagrange equation reads
\[
(-\Delta_{p})^{s} u=\lambda\, |x|^{-\alpha}\, u^{r-1}
\]
with arbitrary $\lambda>0$ when $\alpha<ps$, which we assume.  It is convenient to normalize minimizers  $U$ requiring that 
$\lambda=1$, namely
\begin{equation}\label{EL}
(-\Delta_{p})^{s}U=|x|^{-\alpha}\, U^{r-1}.
\end{equation}
This implies, after testing with $U$,
\begin{equation}\label{norm}
[U]_{s,p}^{p}=|U|_{r,\alpha}^{r}={\cal S}^{\frac{N-\alpha}{p\, s-\alpha}}.
\end{equation}
Finally, observe that for any $\eps>0$ the function
\begin{equation}\label{ueps}
U_{\eps}(x)=\eps^{{\frac{p\, s-N}{p}}}U\left(\frac{x}{\eps}\right)
\end{equation}
is still a minimizer fulfilling \eqref{EL}--\eqref{norm}. Due to \eqref{decay} the family $\{U_{\eps}\}_{\eps}$ concentrates at zero and \eqref{norm} entails
\[
|x|^{-\alpha}\, U^{r}_\eps\ \weakstarto \ {\cal S}^{\frac{N-\alpha}{p\, s-\alpha}}\, \delta_{0}.
\]
Our main result, chiefly based on \cite{MM}, reads as follows.
\begin{theorem} \label{maintheo}
Let $p>1$, $s\in \ ]0, 1[$, $0\le \alpha<ps<N$, and $r$ satisfy \eqref{scal}. Given any positive minimizer $U$ for \eqref{S} fulfilling \eqref{norm}, let $U_{\eps}$ be defined by \eqref{ueps}. Then there exists a family of truncations $\{U_{\eps, \delta}\}_\eps$ of $U_{\eps}$ in $B_{\delta}$ such that,  for every $\eps\le\delta$,
\begin{align}
\label{st1}
q\in \ \left]\frac{N(p-1)}{N-s}, p\right]\quad &\implies\quad [U_{\eps, \delta}]_{s, q}\le  C\, \eps^{\frac{N}{q}-\frac{N}{p}};\\
\label{nu}
q\in \ \left]1, \frac{N(p-1)}{N-s}\right]\quad&\implies\quad\forall\,\nu>0\ \exists\, C_{\nu}:\; [U_{\eps, \delta}]_{s, q}\le  C_{\nu}\, \delta^{\frac{N}{q}-\frac{N}{p}}\, \left(\frac{\eps}{\delta}\right)^{\frac{N-p\, s}{p\, (p-1)}-\nu}.
\end{align}
The constants $C$ and $C_\nu$ are independent of $\eps$ and $\delta$, but may depend on $U$.
\end{theorem}
\noindent
Let us make a few comments on this result.

\begin{description}
\item[Difficulties] As discussed before, a delicate issue peculiar to the fractional setting is that, no matter how smoothly the truncation is implemented, there is no direct way to bound $[U_{\eps, \delta}]_{s, q}$ in terms of $[U_{\eps, \delta}]_{s, p}$. More importantly, even proving that $[U_{\eps, \delta}]_{s, q}$ is finite turns out to be somewhat non-trivial. Indeed, if $q\le p$ then
\[
[U_{\eps}]_{s, q}<+\infty \quad \Leftrightarrow\quad q>\frac{N(p-1)}{N-s};
\]
cf. the introduction of \cite{MM}. Consequently, for $q\le\frac{N(p-1)}{N-s}$, any bound of $[U_{\eps, \delta}]_{s, q}$ in terms of $[U_{\eps}]_{s, q}$ is useless, as the latter is infinite. 
\item[Comparison with the local case] The truncation proposed here can be also performed in the classical framework, i.e., $s=1$. In this case, the minimizers of \eqref{S} are given by \eqref{AT} and an explicit calculation shows 
\begin{equation}\label{fest}
\|\nabla U_{\eps, \delta}\|_{L^{q}(\R^{N})}\le 
\begin{cases}
C\, \eps^{\frac{N}{q}-\frac{N}{p}}&\text{if}\ q\in \ \left]\frac{N(p-1)}{N-1}, p\right],\\[10pt]
C\, \delta^{\frac{N}{q}-\frac{N}{p}}\big(\eps/\delta\big)^{\frac{N-p}{p(p-1)}}&\text{if}\ q\in \ \left]1,\frac{N(p-1)}{N-1}\right[\ ;
\end{cases} 
\end{equation}
see \cite{DH} for similar estimates of truncations via cut-off. Hence, there is a full agreement in the first case and `almost the same' estimate in the other, with the nonlocal bound being slightly worse (but by an arbitrarily small difference from the asymptotic point of view) than the local one. 
\item[Applications] Quantitative truncation estimates reveal particularly useful when critical problems of Br\'ezis-Nirenberg's type are studied. Those involving lower order norms of the gradient naturally arise once the leading term in the equation is of $p$-$q$ Laplacian type, and estimates like \eqref{fest} have had a key r\^ole; see, e.g., \cite{YY, LZ, CMP}. \\
A similar theory has been attempted in recent years for the fractional setting, often based on the assumption that minimizers $U$ of \eqref{S} have a finite $[U]_{s,q}$ semi-norm when $q\le\frac{N(p-1)}{N-s}$. This hypothesis would indeed give estimates fully analogous to the classical case, namely \eqref{nu} with $\nu=0$, but, as already pointed out, it is {\em false}. Nevertheless, we hope that the weaker version \eqref{nu} still suffices to justify most of the results in the literature.
\end{description}

\vskip10pt
\noindent
{\em Notations}: $|A|$ will denote the Lebesgue measure of $A\subseteq \R^{N}$. If $p\ge 1$ and $u:\R^N\to\R$ is measurable then $\|u\|_{L^{p}}=\|u\|_{L^{p}(\R^{N}, dx)}$, provided no confusion can arise. The symbol $C$ will denote a (finite) positive constant, which may change in value from line to line, and whose dependencies are specified when necessary.
\section{Description of truncation and proof of Theorem \ref{maintheo}}
Let $U$ be a normalized minimizer (i.e., obeying \eqref{EL}) and let $U_\eps$ be given by \eqref{ueps}. We will describe a basic truncation technique for
$\{U_{\eps}\}_{\eps}$ first introduced in \cite{MPSY}. The polynomial decay \eqref{decay} reads as 
\[
c_{1} \, \rho^{-\frac{N-p\, s}{p-1}}\le U(\rho)\le  c_{2}\,  \rho^{-\frac{N-p\, s}{p-1}},\quad \rho\ge 1,
\]
where $c_{1}$ and $c_{2}$ depend on $U$. For every $\theta>1$ one infers
\[
\frac{U(\theta \rho)}{U(\rho)}\le \frac{c_{2}}{c_{1}}\, \theta^{\frac{p\, s-N}{p-1}}
\]
so that there exists $\bar\theta$ large such that
\begin{equation}\label{theta}
\frac{U(\bar\theta \rho)}{U(\rho)}\le\frac{1}{2}\, ,\quad\rho\ge 1.
\end{equation}
Set, provided $\eps,\delta>0$,
\[
m_{\eps,\delta}=\frac{U_\eps(\delta)}{U_\eps(\delta) - U_\eps(\bar\theta \delta)}
\]
as well as
\[
G_{\eps,\delta}(t)  = 
\begin{cases}
0 &\text{if }\   0 \le t \le U_\eps(\bar \theta\, \delta),\\[5pt]
m_{\eps,\delta}\, (t - U_\eps(\bar\theta\, \delta)) &\text{if }\   U_\eps(\bar\theta\, \delta) \le t \le U_\eps(\delta),\\[5pt]
t &\text{if }\ t \ge U_\eps(\delta).
\end{cases}
\]
Evidently, the function $G_{\eps, \delta}:\R_{+}\to \R_{+}$ is non-decreasing and absolutely continuous. We define the {\em truncation by composition} of the family $\{U_{\eps}\}_{\eps}$ in $B_{\bar\theta\delta}$  as
\[
U_{\eps,\delta}(\rho) = G_{\eps,\delta}(U_\eps(\rho)),
\]
which is a radially non-increasing function such that
\[
U_{\eps,\delta}(\rho)=
\begin{cases}
U_\eps(\rho) &\text{if }\ \rho\le\delta,\\[5pt]
0 &\text{if }\ \rho\ge\bar\theta\, \delta.
\end{cases}
\]
The following truncation estimates hold true. More general situations are treated in \cite[Lemmas 2.10-2.11]{CSM}.
\begin{lemma}[\cite{Y}, Lemma 2.7] \label{Lemma 3}
There exists a constant $C=C(U, N, p, s)>0$ such that for every $\eps\leq \delta/2$ it holds
\[
[U_{\eps,\delta}]_{s,p}^p \le {\cal S}^{\frac{N-\alpha}{p\, s-\alpha}} + C\, \left(\frac{\eps}{\delta}\right)^{\frac{N-p\,s}{p-1}}
\qquad\text{and}\qquad
\|U_{\eps,\delta}\|_{r,\alpha}^{r}\ge {\cal S}^{\frac{N-\alpha}{p\, s-\alpha}} -  C\, \left(\frac{\eps}{\delta}\right)^{\frac{N-\alpha}{p-1}}.
\]
\end{lemma}
To prove Theorem \ref{maintheo}, some higher differentiability properties at the Besov scale for $U$, essentially contained in \cite{MM}, will be exploited. Let $0<\sigma<2$. The homogeneous Besov semi-norm of a measurable function $v:\R^N\to\R$ is
\[
[v]_{B^{\sigma}_{p, \infty}}:=\sup_{|h|>0}\|h^{-\sigma}{\sf d}^{2}_{h}v\|_{L^{p}},\quad\text{where}\quad
{\sf d}^{2}_{h}v(x)=2\, v(x+h)-v(x)-v(x+2h).
\]
When $\sigma<1$, it is equivalent to the one involving first-order differences, namely
\[
[v]_{{\cal B}^{\sigma}_{p, \infty}}=\sup_{|h|>0}\||h|^{-\sigma}{\sf d}_{h}v\|_{L^{p}},\quad\text{with}\quad
 {\sf d}_{h}v(x)=v(x)-v(x+h).
\]
Indeed, chiefly using $2\,{\sf d}_{h}={\sf d}_{2h}-{\sf d}^{2}_{h}$, one has
\begin{equation}\label{Bequiv}
\frac{1}{2}\, [v]_{B^{\sigma}_{p, \infty}}\le [v]_{{\cal B}^{\sigma}_{p, \infty}}\le \frac{1}{2-2^{\sigma}}\, [v]_{B^{\sigma}_{p, \infty}}.
\end{equation}
\begin{lemma}\label{Beso}
Under the assumptions of Proposition \ref{propmin}, let $U$ be a minimizer for \eqref{S}. Then there exists $\bar\sigma>s$ (depending on $N, p, s, r,\alpha$) such that 
\begin{equation}\label{best}
[U]_{B^{\sigma}_{p, \infty}}<+\infty\quad\forall\, \sigma\in [s,\bar\sigma].
\end{equation}
\end{lemma}
\begin{proof}
By \cite[Lemma 5.6]{MM}, the function $U$ weakly solves $(-\Delta_{p})^{s}U=f$, with $f\in L^{\gamma}(\R^{N})$ for every
$\gamma\in\left[1, \frac{N}{\alpha}\right[$. Proposition \ref{propmin} ensures that $U\in L^{\infty}(\R^{N})$, whence $U\in L^{\beta}(\R^{N})$ for all $\beta\in \ \left]\frac{N(p-1)}{N-p\, s}, +\infty\right]$, as an explicit calculation exploiting \eqref{decay} shows.  We can thus apply the regularity estimate \cite[Lemma 4.3]{MM} to get $[U]_{B^{\sigma}_{p, \infty}}<+\infty$ once
\[
\sigma= 
\begin{cases}
\dfrac{p\, s}{p-\theta}&\text{if $p\ge 2$},\\[10pt]
\dfrac{2\, s}{2-\theta}&\text{if $1<p<2$},
\end{cases}
\quad \text{with $\theta\in \ ]0, 1]$ such that}\quad 
\begin{cases}
\dfrac{\theta}{p}+\dfrac{1-\theta}{\beta}=\dfrac{1}{\gamma'},\\[10pt]
\dfrac{N\, (p-1)}{N-p\, s}<\beta\le +\infty,\\[10pt]
1\le \gamma<\frac{N}{\alpha}.
\end{cases}
\]
The system prescribing possible values of $\theta$ can be explicitly solved, and we arrive at
\[
0\le \theta\le
\begin{cases}
p\left(1-\dfrac{\alpha}{N}\right)&\text{if $\dfrac{1}{p}>\max\left\{\dfrac{N-p\, s}{N\, (p-1)}, 1-\dfrac{\alpha}{N}\right\}$},\\[10pt]
1&\text{otherwise}.
\end{cases}
\]
Letting 
\[
\bar \sigma
=\begin{cases}
\dfrac{p\, s}{p-\bar\theta}&\text{if $p\ge 2$},\\[10pt]
\dfrac{2\, s}{2-\bar\theta}&\text{if $1<p<2$},
\end{cases}
\quad\text{where}\quad\bar\theta=\min\left\{1, p\left(1-\frac{\alpha}{N}\right)\right\}\, ,
\]
the conclusion follows.
\end{proof}
The next elementary lemma will be also employed.
\begin{lemma}
Suppose $0<\sigma<1$ and $1\le q\le p$. Then, for every measurable function $v:\R^N\to\R$ such that ${\rm supp}(v)\subseteq B_{R}$ one has
\begin{equation}\label{Bpq}
[v]_{{\cal B}^{\sigma}_{q, \infty}}\le C(N, p, q)\, R^{\frac{N}{q}-\frac{N}{p}}\, [v]_{{\cal B}^{\sigma}_{p, \infty}}.
\end{equation}
\end{lemma}
\begin{proof}
There is no loss of generality in assuming finite the right hand side of \eqref{Bpq} as well as, after a possible scaling, $R=1$. Observe that from ${\rm supp}(v)\subseteq B_{1}$ we infer
\[
\|{\sf d}_{h} v\|_{L^{r}}=2^{\frac{1}{r}}\, \|v\|_{L^{r}} \quad\text{provided}\;\; |h|\ge 2,\;\; r\geq 1.
\]
since $v(x)$ and $v(x+h)$ have disjoint supports. Via H\"older's inequality, this entails
\[
\|{\sf d}_{h} v\|_{L^{q}}=2^{\frac{1}{q}}\, \|v\|_{L^{q}}\le 2^{\frac{1}{q}}\, |B_{1}|^{\frac{1}{q}-\frac{1}{p}}\, \|v\|_{L^{p}}\leq 2^{\frac{1}{q}-\frac{1}{p}}\, |B_{1}|^{\frac{1}{q}-\frac{1}{p}}\,\|{\sf d}_{h} v\|_{L^{p}}
\]
for $|h|\ge 2$. If $|h|\le 2$ then ${\rm supp}({\sf d}_{h}v)\subseteq B_{4}$. Hence,
\[
\|{\sf d}_{h} v\|_{L^{q}}\le |B_{4}|^{\frac{1}{q}-\frac{1}{p}}\,\|{\sf d}_{h} v\|_{L^{p}}
\]
and taking suprema after multiplying by $|h|^{\sigma}$ completes the proof.
\end{proof}
\begin{proof}[Proof of Theorem \ref{maintheo}]
Consider first the case $\frac{N(p-1)}{N-s}<q\le p$. Inequality \eqref{decD} yields $[U]_{s,q}<+\infty$. Since $G_{\eps, \delta}$ is Lipschitz continuous with constant ${\rm Lip}(G_{\eps, \delta})=m_{\eps, \delta}$ while $m_{\eps, \delta}\le 2$ due to \eqref{theta}, after scaling one has
\[
[U_{\eps, \delta}]_{s, q}\le 2\, [U_{\eps}]_{s, q}=2\, \eps^{\frac{p\, s-N}{p}}\eps^{\frac{N-q\, s}{q}}\, [U]_{s,q}=C_{U}\, \eps^{\frac{N}{q}-\frac{N}{p}},
\]
which shows \eqref{st1}. Let  now $1<q\le\frac{N\, (p-1)}{N-s}$ and let $\bar\sigma$ be given by Lemma \ref{Beso}. Suppose, as we allow, $s<\bar\sigma<1$. If $\sigma\in \ ]s, \bar\sigma]$ then, thanks to \eqref{Bequiv}, the inequality $|{\sf d}_{h}U_{\eps, \delta}|\le 2\, |{\sf d}_{h} U_{\eps}|$ (due to ${\rm Lip}(G_{\eps, \delta})\leq 2$), and a scaling argument, we have
\[
[U_{\eps, \delta}]_{B^{\sigma}_{p, \infty}}\le 2\, [U_{\eps, \delta}]_{{\cal B}^{\sigma}_{p, \infty}}\le 4\, [U_{\eps}]_{{\cal B}^{\sigma}_{p, \infty}}\le \frac{4}{2-2^{\sigma}}\, [U_{\eps}]_{B^{\sigma}_{p, \infty}}\le \frac{4}{2-2^{\bar\sigma}}\, \eps^{s-\sigma}\, [U]_{B^{\sigma}_{p, \infty}}.
\]
Here, $C$ depends on $\bar\sigma$ alone. Consequently,
\[
[U_{\eps, \delta}]_{B^{\sigma}_{p, \infty}}\le C_{U}\,\eps^{s-\sigma},
\]
with $C_{U}<+\infty$ thanks to \eqref{best}. Pick any $t\in\left]\frac{N\, (p-1)}{N-s}, p\right[$. From ${\rm supp} (U_{\eps, \delta})\subseteq B_{\bar\theta\, \delta}$, \eqref{Bpq}, \eqref{Bequiv}, and the above inequality it follows
\begin{equation}\label{last}
[U_{\eps, \delta}]_{B^{\sigma}_{t, \infty}}\le C\, (\bar\theta\delta)^{\frac{N}{t}-\frac{N}{p}}\, [U_{\eps, \delta}]_{B^{\sigma}_{p, \infty}}\le \bar C_{U}\, \delta^{\frac{N}{t}-\frac{N}{p}}\, \eps^{s-\sigma}.
\end{equation}
Thus, Lemma 5.1 in \cite{MM} can be used, with summability exponents $q,t$ and differentiability parameters $s<\sigma$, to achieve 
\[
[U_{\eps, \delta}]_{s, q}\le C\, \delta^{\frac{N}{q}-\frac{N}{t}+\mu\, (\sigma-s)}\, [U_{\eps, \delta}]_{B^{\sigma}_{t, \infty}}^{\mu}\, [U_{\eps, \delta}]_{s, t}^{1-\mu}\quad\forall\,\mu\in\  ]0, 1[\, ,
\]
where $C$ depends on all parameters involved, except $\eps$ and $\delta$. The choice of $t$ ensures that we can estimate $[U_{\eps, \delta}]_{s, t}$ through \eqref{st1}, while \eqref{last} bounds $[U_{\eps, \delta}]_{B^{\sigma}_{t, \infty}}$, so that both are finite. Summing up, one has
\[
\begin{split}
[U_{\eps, \delta}]_{s, q}&\le \bar C_{U}\, \delta^{\frac{N}{q}-\frac{N}{t}+\mu\, (\sigma-s)}\, \delta^{\mu\, \big(\frac{N}{t}-\frac{N}{p}\big)}\, \eps^{\mu\, (s-\sigma)}\eps^{(1-\mu)\, \big(\frac{N}{t}-\frac{N}{p}\big)}\\
&=\bar C_{U}\, \delta^{\frac{N}{q}-\frac{N}{p}}\, \left(\frac{\eps}{\delta}\right)^{(1-\mu)\, \big(\frac{N}{t}-\frac{N}{p}\big)-\mu\, (s-\sigma)}
\end{split}
\]
with $\bar C_{U}$ depending on all the parameters except $\eps$ and $\delta$.
This shows \eqref{nu}: indeed, the exponent $(1-\mu)\, \big(\frac{N}{t}-\frac{N}{p}\big)-\mu\, (s-\sigma)$ turns out to be always less than
$\frac{N-p\, s}{p\, (p-1)}$ but can be made arbitrary close to it by simply choosing $\mu$ and $t-\frac{N\,(p-1)}{N-s}$ small enough.
\end{proof}


\begin{thebibliography}{99}

\bibitem{CMP}
{\sc P.\ Candito, S.\,A.\ Marano, K.\ Perera},
On a class of critical $(p,q)$-Laplacian problems, 
{\em NoDEA Nonlinear Differential Equations Appl.} {\bf 22} (2015), 1959--1972.

\bibitem{CSM}
{\sc W.\ Chen, S.\ Mosconi, M.\ Squassina},  Nonlocal problems with critical Hardy nonlinearity, {\em J. Funct. Anal.} {\bf 275} (2018), 3065--3114.

\bibitem{DH}
{\sc P.\ Dr\'abek, Y.\,X.\ Huang},
Multiplicity of positive solutions for some quasilinear elliptic equation in $\R^N$ with critical Sobolev exponent,
{\em J. Differential Equations} {\bf 140} (1997), 106--132.



\bibitem{LZ}
{\sc G.\ Li, G. Zhang}, Multiple solutions for the $p$\&$q$-Laplacian problem with critical exponent, {\em Acta Math. Sci. Ser. B}, Engl. Ed. {\bf 29} (2009), 903--918.


\bibitem{MM1} 
{\sc S.A. \ Marano, S.\ Mosconi},
Some recent results on the Dirichlet problem for $(p,q)$-Laplace equations,
{\em Discrete Contin. Dyn. Syst. Ser. S} {\bf 11} (2018), 279--291.

\bibitem{MM}
{\sc S.A.\ Marano, S.\ Mosconi}, 
Asymptotics for optimizers of the fractional Hardy-Sobolev inequality,
{\em Commun. Contemp. Math.} {\bf 21} (2019), 1850028.

\bibitem{MS}
{\sc P.\ Mironescu, W.\ Sickel},
A Sobolev non embedding,
{\em Atti Accad. Naz. Lincei Rend. Lincei Mat. Appl.} {\bf 26} (2015), 291--298.

\bibitem{MPSY}
{\sc S.\ Mosconi, K.\ Perera, M.\ Squassina, Y.\ Yang},
The Brezis-Nirenberg problem for the fractional $p$-Laplacian, 
{\em Calc. Var. Partial Differential Equations} {\bf 55} (2016), 55:105.

\bibitem{MSP}
{\sc S.\ Mosconi, M.\ Squassina},
Recent progresses in the theory of nonlinear nonlocal problems, {\em Bruno Pini Mathematical Analysis Sem.} {\bf 7} (2016), 147--164.

\bibitem{Y}
{\sc Y.\ Yang}, The Brezis-Nirenberg problem for the fractional $p$-Laplacian involving critical Hardy- Sobolen exponents, Arxiv preprint 1710.04654.

\bibitem{YY}
{\sc Z.\ Yang, H.\ Yin},
Multiplicity of positive solutions to a $p-q$-Laplacian equation involving critical nonlinearity,
{\em Nonlinear Anal.} {\bf 75} (2012), 3021--3035.

\end{thebibliography}
\end{document}